\documentclass[a4paper,10pt]{elsarticle}

\usepackage[margin=3cm]{geometry} 
\usepackage{amsmath, amsfonts, amssymb, amsthm}
\usepackage{graphicx,enumerate}
\usepackage{array,multirow,hhline}

\newtheorem{theorem}{Theorem}

\begin{document}
\title{A rapidly converging domain decomposition method for the 
Helmholtz equation}
\author{Christiaan C. Stolk}
\ead{C.C.Stolk@uva.nl}
\address{University of Amsterdam, 
Korteweg-de Vries Institute for Mathematics,
P.O.Box 94248, 1090 GE Amsterdam, The Netherlands}

\begin{abstract}
\noindent
A new domain decomposition method is introduced for the heterogeneous
2-D and 3-D Helmholtz equations. Transmission conditions based on
the perfectly matched layer (PML) are derived that avoid artificial
reflections and match incoming and outgoing waves at the subdomain
interfaces.  We focus on a subdivision of the rectangular domain into
many thin subdomains along one of the axes, in combination with a
certain ordering for solving the subdomain problems and a GMRES outer
iteration. 
When combined with multifrontal methods, the solver has near-linear
cost in examples, due to very small iteration numbers that are
essentially independent of problem size and number of subdomains.
It is to our knowledge only the second method with this property next
to the moving PML sweeping method.
\end{abstract}

\begin{keyword}
Helmholtz equation \sep
domain decomposition \sep 
perfectly matched layers \sep
high-frequency waves
\MSC[2010] 65N55, 65N22 
\end{keyword}

\maketitle

\section{Introduction}

In this paper we introduce a new domain decomposition method for the 
solution of the Helmholtz equation in two and three dimensions. 
To be specific we consider in 2-D
\begin{equation} \label{eq:helm1}
  - \partial_{xx}^2 u (x,y) - \partial_{yy}^2 u(x,y) 
  - k(x,y)^2 u(x,y) = f(x,y) ,
\end{equation}
where $k(x,y) = \frac{\omega}{c(x,y)}$, with $c(x,y)$ the wave speed.
The computational domain is assumed to be a rectangle that is truncated 
using the perfectly matched layer \cite{Berenger1994}.
We focus on solving the large linear systems resulting from discretization
with standard 5 or 7 point finite differences. 

We have two main findings. First, we have constructed new
transmission conditions. These are designed to ensure that
\begin{equation} \label{eq:conditions}
\parbox{13cm}{%
\begin{enumerate}[(i)]
\item
the boundary conditions at the subdomain interfaces are 
non-reflecting; 
\item
if $\Omega_{j-1}$ and $\Omega_j$ are neighboring subdomains then
the outgoing wave field from $\Omega_{j-1}$ equals the  
incoming wave field in $\Omega_j$ at the joint boundary and vice versa.
\end{enumerate}}
\end{equation}
This is achieved in a simple and accurate way using PML boundary layers
added to the subdomains and single layer potentials. See 
\cite{SchadleEtAl2007} for a related approach in the finite element 
discretization of the time harmonic Maxwell equations.

Our most remarkable finding concerns the situation where the domain is
split into many thin layers along one of the axes, say $J$ subdomains
numbered from 1 to $J$. Following \cite{EngquistYing2011} we will also
call these quasi 2-D subdomains.  Generally, an increase in the number of
subdomains leads to an increase in the number of iterations required
for convergence. Here we propose and study a method where {\em the number
of iterations is essentially independent of the number of subdomains}.

A necessary condition for this is that information can travel over the
entire domain (or at least an $O(1)$ part thereof) in one
iteration.  To achieve this we use a multiplicative method, where the
subdomains are first solved consecutively from $j=1$ to $j=J$, each
time using information from the solution of the neighboring,
previously solved subdomain, and then in the same way from $j=J$
downto $j=1$ using the residual as right hand side. In this way
information can travel all over the domain with only two solves per
subdomain. The procedure is used as a preconditioner for GMRES.

We studied numerically the convergence of the method for
different choices of the grid distance $h$ and the frequency $\omega$,
keeping $\omega h$ constant, and different numbers of subdomains.
In our examples the method converged rapidly, with generally less
than 10 iterations needed for reduction of the residual by
$10^{-6}$.  Moreover, the required number of iterations was essentially
independent of the size of the domain and the number of subdomains.

This is attractive in combination with the use of multifrontal methods
for the subdomain solves. Indeed, as was argued by Engquist and Ying 
\cite{EngquistYing2011}, in 3-D the set of quasi 2-D subproblems can 
be solved by multifrontal methods in $O(N \log N)$ time, with 
$O(N^{4/3})$ cost for the factorization, versus $O(N^{3/2})$ and $O(N^2)$ 
when the multifrontal method is applied directly to the 3-D system.  
The method therefore behaves near-linearly%
\footnote{meaning linearly if $\log N$ factors are ignored}.
It is the second such method we are aware of, in addition to the moving 
PML sweeping method, for which such observations were made 
in \cite{EngquistYing2011}.

Several things have to be kept in mind. 
First, we estimate, based on our examples, that the
thickness of the PML layers needs to increase with increasing $N$.  A
required growth of $O(\log N)$ is consistent with our data.  This
would lead to an additional factor $O(\log N)$ for the cost of the
solves and $O((\log N)^2)$ for the cost of the preparation. Secondly,
the method is only near-linear provided that solutions are required
for a sufficiently large number of right hand sides to recoup the cost
of the factorization.
Thirdly, because of the multiplicative way of domain decomposition,
the method is not in itself parallel. In section~\ref{sec:conclusions} 
two solutions for this problem are discussed. Finally, numerical
tests have shown that for cavities, the claimed results do not hold. 
(Indeed, the cavity problem is known to be especially difficult for
iterative methods because of the many near-zero eigenvalues.)

\subsection{The method and its context}

Next we discuss in more detail the ideas behind the method and some of the 
relevant literature.

To motivate our approach we recall the 1-D problem with $k=\text{constant}$,
see the review in \cite{CollinoGhanemiJoly2000} or 
\cite{Despres1991_Waves1991,Despres1990,ShaidurovOgorodnikov1991}.
Let $]0,L[$ be the domain. The differential equation and the Robin 
boundary conditions  read
\begin{align*}
- \partial^2_{xx} u(x) -k^2 u(x) = {}& f(x)
&&
  \text{for $0< x < L$},
\\
  \partial_x u + iku={} & 0 
&&
  \text{at $x=0$}
\\
  - \partial_x u + iku={} & 0 
&&
  \text{at $x=L$.}
\end{align*}
The Robin boundary conditions are exact non-reflecting boundary
conditions and ensure that there are no incoming waves at the boundaries.
We assume the domain is divided in $J$ subdomains 
$]b_{j-1},b_j[$ with
\[
  0 = b_0 < b_1 < \ldots < b_J = L .
\]
The original problem is then equivalent to $J$ subdomain problems 
with continuity conditions at the interfaces as follows 
\begin{align*}
  -\partial_{xx}^2  u^{(j)} - k^2 u^{(j)} = {}& f^{(j)} 
&& \text{for } x \in ]b_{j-1},b_j[
\\[1ex]
  \partial_x u^{(j)} + ik u^{(j)} ={}&
  \partial_x u^{(j-1)} + ik u^{(j-1)} 
&&\text{at } x=b_{j-1} 
\\
  - \partial_x u^{(j)} + ik u^{(j)} = {}&
  - \partial_x u^{(j+1)} + ik u^{(j+1)}
&&\text{at } x=b_{j} 
\end{align*}
(by convention $u^{(0)} = 0 = u^{(J+1)}$). 
These continuity conditions satisfy the property (\ref{eq:conditions}).
To obtain an iterative solution method, the right hand side of the
continuity conditions is taken from the previous iteration, i.e.\
a sequence $v_n^{(j)}$ is constructed, where $n$ is the iteration number
and $j$ the subdomain index according to
\begin{align} \label{eq:iterativemethod1D-1}
  -\partial_{xx}^2  v_n^{(j)} -k^2 v_n^{(j)} = {}& f^{(j)} 
&& \text{for } x\in ]b_{j-1},b_j[
\\  \label{eq:iterativemethod1D-2}
  \partial_x v_n^{(j)} + ik v_n^{(j)} ={}&
  \partial_x v_{n-1}^{(j-1)} + ik v_{n-1}^{(j-1)} ,
&& \text{at } x=b_{j-1}
\\ \label{eq:iterativemethod1D-3}
  - \partial_x v_n^{(j)} + ik v_n^{(j)} = {}& 
  - \partial_x v_{n-1}^{(j+1)} + ik v_{n-1}^{(j+1)}
&& \text{at } x=b_{j} .
\end{align}
This method is optimal in the sense that it converges in a finite
number, namely $J$, of iterations.
Indeed, recall that the solution for the 
problem $-\partial_{xx}^2 u(x) - k^2 u(x) = f(x)$ with
Robin boundary conditions
$  \partial_x u(0) + iku(0) = h_1$,
$- \partial_x u(L) + iku(L) = h_2$ 
is given by
\begin{equation} \label{eq:solutionformula1D}
  u(x) = \frac{i}{2k} \int_{0}^x e^{ik(x-s)} f(s) \, ds 
+ \frac{i}{2k} \int_x^L e^{-ik(x-s)} f(s) \, ds
+ \frac{e^{ikx}}{2ik} h_1
+ \frac{e^{-ik(x-L)}}{2ik} h_2 
\end{equation}
It follows by induction, starting from $v_0^{(j)}= 0$,
that $v_n^{(j)}$ satisfies
\[
  v_n^{(j)}(x) = \frac{i}{2k} \int_{A}^x e^{ik(x-s)} f(s) \, ds 
+ \frac{i}{2k} \int_x^B e^{-ik(x-s)} f(s) \, ds
\]
where $A = b_{\max(0,j-n)}$ and $B= b_{\min(J,j+n-1)}$. After $J$ steps,
$A = 0$ and $B=L$ for all $j \in \{1,\ldots,J\}$.

This work answers two questions about the iterative method
(\ref{eq:iterativemethod1D-1})-(\ref{eq:iterativemethod1D-3}).
The first question concerns the generalization of 
the transmission conditions to two 
and three dimensions. The Robin boundary conditions
then no longer satisfies the properties  (\ref{eq:conditions})
(see the argument around (\ref{eq:Fourier_transformed_PDE_2D}) below).
Several approaches have been proposed in the literature. 
First, the Robin boundary conditions can still be used as transmission 
conditions \cite{BenamouDespres1997}. 
Several authors have also considered optimized Robin transmission 
conditions \cite{GanderMagoulesNataf2002,GanderHalpernMagoules2007}.
A second possible approach involves operator valued Robin boundary
conditions \cite{NatafRogierdeSturler1994_Preprint} and ideas about 
numerical absorbing boundary conditions, e.g.\ \cite{EngquistMajda1977}. 
Pad\'e approximations for $\lambda$ in
(\ref{eq:perfectly_non_reflecting2}) (see below) can be used
to obtain numerical absorbing boundary and transmission conditions 
\cite{ChevalierNataf1998,BoubendirAntoineGeuzaine2012}.
In this paper we use PML boundary layers 
\cite{Berenger1994} to achieve (\ref{eq:conditions}).
Earlier work using domain decomposition with PML's is in
\cite{Toselli1998} and in \cite{SchadleEtAl2007} (cf.\
the discussion in section~\ref{sec:conclusions}).

The second question concerns the case of large $J$. 
In one iteration of 
(\ref{eq:iterativemethod1D-1})-(\ref{eq:iterativemethod1D-3})
information from one subdomain can only travel 
to its neighbors. The method therefore requires at least
$O(J)$ iterations to converge, hence $O(J^2)$ subdomain solves.
On the other hand, by using the multiplicative approach outlined
below, solving the subdomains consecutively, first
$j=1,2,\ldots, J$, and then $j=J,J-1, \ldots, 1$, information
can travel over the full domain in just $2$ solves per subdomain. Here 
we will follow this multiplicative approach.

As mentioned, the case of a large number of thin layers, say 
$k$ grid points thick, is of interest when the method is 
used in combination with multifrontal methods for the subdomain
solves.The computational cost of such a setup was analyzed by
Engquist and Ying \cite{EngquistYing2011}, using results
of \cite{George1973}.  Consider a cube with $n
\times n \times n$ gridpoints, hence $N =n ^3$.  The cost of a $LU$
decomposition of a subdomain of the form $n \times n \times k$ is
$O(k^3 n^3)$, while the cost of a backsubstition is $O(k^2 n^2 \log
n)$. Assuming $k = O(1)$, the total cost of the factorizations is
$O(N^{4/3})$, while the total cost per iteration is $O(N \log N)$.  
If the number of iterations depends weakly on problem size, as we see in
examples, then this method scales well. In the presence
of PML layers that have a thickness of $w_{\rm pml}$ grid points, 
a value $k \approx 4 w_{\rm pml}$ is optimal for the thickness of
the subdomains including PML layers, i.e.\ minimizes the cost of applying
one set of subdomain solves.
The details of our method will be explained in section~\ref{sec:themethod}.

The papers \cite{Erlangga2008,ErnstGander2011} provide 
a review of solution methods for the Helmholtz equation.
Some recent other work is given in  \cite{WangDeHoopXia2011,
BollhoferGroteSchenk2009}.

\subsection{Results}

Our first main result is a theoretical result, concerning the constant
coefficient problem on a strip. Assuming that the PML layers
perfectly reproduce the behavior of the solution on the unbounded
domain, the methods solves this problem in one iteration, i.e.\ in one
upward and one downward sequence of solves. We observe that the upward
and downward sequence of solves can in fact be performed simultaneously,
if the point where the the sequences cross is handled carefully.

The second main result is the good convergence behavior in numerical 
examples that was already mentioned in the first part of this
introduction. In addition a comparison with a double sweep method 
with Robin transmission conditions was made. For small $J$ this 
can be attractive, but this method did not have near-linear cost
like the PML-based method.

\subsection{Contents}

The paper is organized as follows. The next section explains in detail
our method. Section~\ref{sec:theoreticalresults} contains some theoretical 
results. Then in section~\ref{sec:numericalresults} the numerical examples
are discussed. We end the paper with a short discussion.

\section{The method}
\label{sec:themethod}

\subsection{Continuous formulation}

In this section we formulate our method in 2-D. The domain is assumed
to be a set of the form $\Omega = ]0,L[ \times ]0,1[$. It is 
straightforward to generalize this to rectangular domains of different size,
and to 3-D rectangular domains. 

The Helmholtz operator will be referred to as $A$, given away from the 
PML boundary layers by
\[
  A = -\partial^2_{xx} - \partial^2_{yy} - k(x,y)^2 .
\]
The operator in a PML layer at a boundary, say $x=\text{constant}$, is
obtained by replacing
\[
  \frac{\partial}{\partial x} 
\rightarrow
  \frac{1}{1+ i \frac{\sigma_x(x)}{\omega} }
  \frac{\partial}{\partial x} 
\]
where $\sigma_x = 0$ in the interior of the domain, and positive inside 
the PML layers \cite{ChewWeedon1994,Johnson_notespml}.

The domain is divided into $J$ subdomains along the $x$-axis. The interface
locations will be denoted by $x = b_j$, where
\[
  0=b_0 < \ldots < b_J = L
\]
The ``core'' subdomains, without additional PML layers, will be denoted 
by $D^{(j)} = ] b_{j-1}, b_j [ \times ]0,1[$. With PML layers added the 
notation $\Omega^{(j)}$ will be used. The latter sets are obtained by
padding the $D^{(j)}$ with PML layers of size $L_{\rm pml}$ at the 
internal boundaries, i.e.
\[
  \Omega^{(j)} = ] b_{j-1} - L_{\rm PML} (1-\delta_{j,1}),
b_j + L_{\rm PML} ( 1- \delta_{j,J}) [        \times ]0,1[ 
\]
On the domains $\Omega^{(j)}$, functions $k^{(j)}(x,y)$ are defined
that agree with $k$ on $D^{(j)}$, and are independent of $x$ and 
equal to $k$ at the boundary of the core subdomain inside the added 
PML layers, i.e.\
\[
  k^{(j)}(x,y) = 
\left\{ \begin{array}{ll}
  k(x,y)       & \text{for $b_{j-1} \le x \le b_j$} \\
  k(b_{j-1},y) & \text{for $x < b_{j-1}$ (if $j>1$)} \\
  k(b_j,y)    & \text{for $x > b_{j}$ (if $j<J$).}
\end{array} \right.
\]
On the domains $\Omega^{(j)}$ operators $A^{(j)}$ are defined as Helmholtz
operators with PML modifications, similar as $A$ was defined on $\Omega$.

Next we consider the approximation by domain decomposition of a solution
$u$ to the 2-D Helmholtz equation $A u = f$. The function $f$ is assumed
to be integrable, which allows the definition of $f^{(j)}$ on $\Omega^{(j)}$ 
by
\[
  f^{(j)} = \left\{\begin{array}{ll}
        f(x)  & \text{if $x \in D^{(j)}$}\\
        0     & \text{otherwise}
  \end{array}\right.
\]
A first set of subdomain solutions $v^{(j)}$ is obtained by solving the 
equations
\begin{equation} \label{eq:solve_upward}
  A^{(j)} v^{(j)}
  = f^{(j)} - 2 \delta(x-b_{j-1}) \partial_x v^{(j-1)}(b_{j-1},\cdot) ,
\end{equation}
consecutively for $j=1,2,\ldots,J$. Here by convention $v^{(0)} = 0$. 
A function $v$ on $\Omega$ is then defined by
\begin{equation} \label{eq:v_from_vj}
  v(x,y) = v^{(j)}(x,y) \qquad \text{with $j$ s.t.\ $b_{j-1} < x < b_j$.}
\end{equation}

The second term in the right hand side of (\ref{eq:solve_upward}) 
requires some explanation. While this is mostly done in the next 
section, a short intuitive explanation goes as follows.
The term $v^{(j-1)}(b_{j-1},\cdot)$ exclusively contains forward going waves 
because of the presence of a PML non-reflecting layer immediately to its right
in $\Omega^{(j)}$. The term 
$- 2 \delta(x-b_{j-1}) \partial_x v^{(j-1)}(b_{j-1},\cdot)$ is meant
to cause the same forward going wave field in the field $v^{(j)}$.
The form of this term can be explained by the properties of the single 
layer potential. The solution to
\[
  A u  = h(y) \delta(x-b_{j-1})
\]
has the property that 
\[
  \lim_{\epsilon \rightarrow 0}
  \partial_x u(b_{j-1}+\epsilon,y) 
  - \partial_x u(b_{j-1} - \epsilon,y)
  = - h(y) ,
\]
if $k$ is continuous at $x=b_{j-1}$.
Assuming the medium $k^{(j)}$ is independent of $x$, the source 
$h(y) \delta(x-b_{j-1})$ generates waves propagating both forwardly
and backwardly in a symmetric fashion. The factor $-2$ is introduced
so that the forward propagating part equals $v^{(j-1)}(b_{j-1},y)$.
The backward propagating part is absorbed in the neighboring PML layer.
Note that in this way, all the subdomain sources $f^{(k)}$ with
$k \le j$ can contribute to the field $v^{(j)}$.

The downward sequence of subdomain solves takes as right hand side 
the restrictions to a subdomain of the residual 
\[
  g = f - A v
\]
However, $v$ is undefined and generally discontinuous at the boundaries 
$x=b_j$, $j=1,\ldots, J-1$. While $g$ is still well defined, it only
exists as a generalized function (distribution), with most singular
term of the form $\delta'(x-b_j) h(y)$.

The problem with this is {\em not} that $g$ is unsuitable as a right
hand side. Solutions to Helmholtz equations with distributional
right hand sides in general exist. And, as a Helmholtz equation is 
formally an elliptic equation, the solutions are smooth away from the 
singular support of the right hand side. However, the restriction of 
$g$ to the subdomains $D^{(j)}$ is not well defined. Indeed, such 
a restriction is obtained by multiplying $g$ by the indicator function
$I_{D^{(j)}}$ of $D^{(j)}$, and this multiplication is in general
not well defined, because of the overlapping singular supports.

Therefore we introduce a {\em second set of domain boundaries}
\[
  0=\tilde{b}_0 < \tilde{b}_1 < \ldots < \tilde{b}_{\tilde{J}} = L .
\]
with $\tilde{b}_j \neq b_k$ for all $0 < j < \tilde{J}$ and 
$0 < k < J$. Similarly as above we defined sets $\tilde{D}^{(j)}$
and $\tilde{\Omega}^{(j)}$, by 
$\tilde{D}^{(j)} = ] \tilde{b}_{j-1}, \tilde{b}_j [ \times ]0,1[$,
and $\tilde{\Omega}^{(j)} = ] \tilde{b}_{j-1} - L_{\rm PML} (1-\delta_{j,1}),
\tilde{b}_j + L_{\rm PML} ( 1- \delta_{j,J}) [ \times ]0,1[$,
and we let $\tilde{A}^{(j)}$ be the Helmholtz operator
with PML modification on $\tilde{\Omega}^{(j)}$.
The function $g^{(j)}$ on $\tilde{\Omega}^{(j)}$ can now be defined by
\[
  g^{(j)}(x,y) = I_{\tilde{D}^{(j)}} g(x,y) .
\]

Next a series of functions $w^{(j)}$ on 
$\tilde{\Omega}^{(j)}$ is determined for $j = \tilde{J},
\tilde{J}-1, \ldots, 1$ (computed in this order) from the equations
\begin{equation} \label{eq:solve_downward}
  \tilde{A}^{(j)} w^{(j)}
  = g^{(j)} + 2 \delta(x-\tilde{b}^{(j)})
        \partial_x w^{(j+1)}(\tilde{b}_{j} , \cdot) ,
\end{equation}
and a function $w$ is defined by 
\[
  w(x,y) = w^{(j)}(x,y)
\]
where $j$ is such that $\tilde{b}_{j-1} < x < \tilde{b}_j$. The function
$w$ is in general undefined for $x = \tilde{b}_j$, $1 \le j \le \tilde{J}-1$, 
where it is discontinuous, but this is not a problem (we don't go into
detail regarding the regularity of the solutions in this work.)

The approximate solution to the Helmholtz equation is given by
$v + w$. We define $P$ to be the map $f \mapsto v+w$.
The map $P$ can be used as a left- or right preconditioner in an iterative
solution method like GMRES.

\subsection{Discrete formulation}

Discretization is done using finite differences. We focus on a
relatively simple scheme, using the standard second order
approximation to the Laplacian. Because the emphasis in this work is
on the convergence of the iterative method for the discrete system,
and on a proof of principle, questions related to the use of higher
order discretizations and the use of different schemes such as finite
elements are relegated to later work.

The grid distance is assumed to be equal in $x$ and $y$ directions
and is denoted by $h$. The grid size is denoted by $n_x \times n_y$.
In 2-D, the finite difference approximation to $A u$ is given by
\[
  (Au)_{i,j} 
=
  \frac{1}{h^2} \left( - u_{i-1,j} + 2 u_{i,j} - u_{i+1,j} \right)
  + \frac{1}{h^2} \left( - u_{i,j-1} + 2 u_{i,j} - u_{i,j+1} \right)
  - k_{i,j}^2 u_{i,j}
\]
In 3-D we use
\[
\begin{split}
  (Au)_{i,j,k} 
= {}&
  \frac{1}{h^2} \left( - u_{i-1,j,k} + 2 u_{i,j,k} - u_{i+1,j,k} \right)
  + \frac{1}{h^2} \left( - u_{i,j-1,k} + 2 u_{i,j,k} - u_{i,j+1,k} \right)
\\
{}&  + \frac{1}{h^2} \left( - u_{i,j,k-1} + 2 u_{i,j,k} - u_{i,j,k+1} \right)
  - k_{i,j,k}^2 u_{i,j,k}
\end{split}
\]
In the PML layers we use the approximation
\[
  \alpha_x \partial_x ( \alpha_x \partial_x u(x_i,y_j) )
=
  \alpha_x(x_i) \frac{ \alpha_x(x_{i+1/2}) \frac{u_{i+1,j} - u_{i,j} }{h}   
        - \alpha_x(x_{i-1/2}) \frac{u_{i,j} - u_{i-1,j}}{h} }{h}
\]
where $\alpha_x(x) = \frac{1}{1 + i \frac{\sigma_x(x)}{\omega}}$.
The subdomain boundaries are assumed to be at half grid points
$b_j = x_{\beta_j + 1/2}$. The discrete equivalent to the interval
$]b_{j-1},b_j[$ is therefore the set of points
$\{ x_{\beta_{j-1}+1} , \ldots, x_{\beta_j} \}$.

The use of two sets of subdomains, with two sets of $L DL^t$ factorizations
of the $A^{(j)}$ is not very attractive. Fortunately
it is not needed. After the first set of discrete subdomain 
boundaries $\beta_j$ is chosen, the second set is defined by
\[
\begin{split}
  \tilde{\beta}_0 = {}& \beta_0
\\
  \tilde{\beta}_J = {}& \beta_J
\\
  \tilde{\beta}_j = {}& \beta_j + 1
\qquad  \text{for  $j=1,\ldots, J-1$}.
\end{split}
\]
The domain for the operators $A^{(j)}$ is given by the grid
\begin{equation} \label{eq:grid_A_j}
\begin{split}
  &\{ x_{\beta_{j-1} + 1} , \ldots, x_{\tilde{\beta}} \}
   \times \{ y_1 ,\ldots, y_{n_y} \} \quad \text{ extended with 
   PML layers of thickness $w_{\rm pml}$}
\\
&\text{on the internal 
  boundaries.}
\end{split}
\end{equation}

Finally, we need to specify the derivative $\partial_x$ and the 
distribution $\delta(x - b_j)$ on the right
hand side of (\ref{eq:solve_upward}) and (\ref{eq:solve_downward})
We approximate derivative on a half-grid point by
\[
  \partial_x u(x_{j+1/2}) \approx \frac{1}{h} ( u_{j+1} - u_j) .
\]
The $\delta$ function is approximated by
\[
  \delta(x_l - x_{j+1/2}) \approx
\left\{ \begin{array}{ll} 
  \frac{1}{2h} & \text{if } |l-(j+1/2)| = 1/2 \\
  0            & \text{otherwise}
\end{array} \right.
\]

We generally aim that all subdomains have approximately the same size
in number of gridpoints. Since this size is given by
$(n_x + (J-1)(2 w_{\rm pml} + 1)) n_y$, we choose the $\beta_j$ such that
\begin{equation} \label{eq:define_beta_j}
  \beta_j \approx w_{\rm pml} + j \frac{n_x - 2 w_{\rm pml} - 1}{J}
\end{equation}

\subsection{Algorithm}

In the previous sections the operators $A$ and $P$ where 
specified. Our plan is to use GMRES for one of the following 
two equations, the right preconditioned system
\begin{equation} \label{eq:left_preconditioned}
  AP v = f , \qquad \qquad u = P v
\end{equation}
or the left preconditioned system
\begin{equation} \label{eq:right_preconditioned}
  PA u = Pf
\end{equation}
These appear to be systems of size $n_x n_y \times n_x n_y$,
but as is common in domain decomposition methods, a modification
of the problem to one involving only degrees of freedom near the 
boundary is possible at least for (\ref{eq:right_preconditioned}).

Indeed, the GMRES iteration of the right-preconditioned problem 
can straightforwardly be restricted to the $2(J-1)$ layers
of grid points at $x_{k,l}$, $k=\beta_j+1$ and $k=\beta_j+2$,
for $j=1,\ldots, J-1$. This is based on two observations. The first is 
that for any $f$, the residual $f - AP f$ is only non-zero
at grid points $x_{k,l}$ with $k = \beta_j+1$ or $k = \beta_j+2$,
for some $j$, $1 \le j \le J-1$.
The second is that the right hand side $f$ can easily be replaced
by a right hand side $\phi$ with the same property.
Namely, let $\tilde{f}^{(j)}$ be restriction of $f$ to the $x_{k,l}$ with
$k \in \{ \tilde{\beta}_{j-1}+1 ,\ldots, \tilde{\beta}_j \}$, and
$\tilde{u}^{(j)}$ be the solution to
\begin{equation} \label{eq:reduce_compute_ut}
  A^{(j)} \tilde{u}^{(j)} = \tilde{f}^{(j)}
\end{equation}
and set $\tilde{u}_{k,l} = \tilde{u}^{(j)}_{k,l}$
if $k \in \{ \tilde{\beta}_{j-1}+1 ,\ldots, \tilde{\beta}_j \}$.
Then as new right hand side the residual can be used
\[
  \phi = f - A \tilde{u} .
\]
Then, after solving $\psi$ from
\begin{equation} \label{eq:solve_A_reduced}
   A \psi = \phi
\end{equation}
using the right-preconditioned equation in the reduced space,
the solution of the original problem is obtained by taking
\[
  u = \psi + \tilde{u} .
\]

This concludes our outline of the method.
In Table~\ref{tab:list_of_algorithms} the main steps that can
be used in a computer implementation are outlined.

\newlength{\somelengthI}
\setlength{\somelengthI}{0.94\textwidth}
\begin{table}
\begin{center}
\framebox{\begin{minipage}{\somelengthI}
\newcommand{\mytab}{\hspace*{2.5mm}}
{\sc Algorithm 1}: Preparation\\
\mytab 
given $J,n_x$ determine subdomain boundaries from (\ref{eq:define_beta_j})\\
\mytab
create matrix of operators $A^{(j)}$ on subdomains given in 
(\ref{eq:grid_A_j})\\
\mytab
perform $L D L^t$ decomposition

\medskip

{\sc Algorithm 2}: Transform to have data only at boundaries\\
\mytab
solve $\tilde{u}^{(j)}$ and $\tilde{u}$ from (\ref{eq:reduce_compute_ut})\\
\mytab
output $\phi = f - A \tilde{u}$

\medskip

{\sc Algorithm 3}: Apply $AP$ to an input $f$\\
\mytab
for $j=2 , ... , J$, solve (\ref{eq:solve_upward}) \\
\mytab
compute the residual $g=f - A v$\\
\mytab
for $j=J-1,\ldots, 1$, solve (\ref{eq:solve_downward})\\
\mytab
compute the residual $h=g - A w$\\
\mytab
output $f-h$

\medskip

{\sc Algorithm 4}: Apply $P$ to an input $\psi$ and add $\tilde{u}$\\
\mytab
for $j=2 , ... , J$, solve (\ref{eq:solve_upward}) with $f=\psi$\\
\mytab
compute the residual $g=\psi - A v$\\
\mytab
for $j=J-1,\ldots, 1$, solve (\ref{eq:solve_downward})\\
\mytab
output $v+w + \tilde{u}^{(j)}$

\medskip

{\sc Algorithm 5}: Solve $A u = f$\\
\mytab
use algorithm 2 to compute $\phi$\\
\mytab
apply GMRES with $AP$ given by algorithm 3 to solve 
(\ref{eq:solve_A_reduced})\\
\mytab
use algorithm 4 to compute solution $u$ from $\psi$
\end{minipage}}
\end{center}
\caption{List of algorithms. Algorithms 1 and 5 form the top-level part of 
the program.}
\label{tab:list_of_algorithms}
\end{table}

\section{Theoretical results}
\label{sec:theoreticalresults}

\subsection{Multiplicative domain decomposition with upward and downward 
sweeps in 1-D}

Here we study our approach of using upward and downward sweeps of
subdomain solves for the 1-D problem. We establish that the constant
coefficient 1-D problem is solved in one step with this method. A
similar result holds when the upward and downward sequences of solves
are done concurrently. Note that these results are different from
those in \cite{NatafRogierdeSturler1994_Preprint}, even though similar
ideas are used in the proofs.

For the upward sweep, consider $v^{(j)}$ defined by
\begin{align} 
  -\partial_{xx}^2  v^{(j)} -k^2 v^{(j)} = {}& f^{(j)} 
&& \text{for } x\in ]b_{j-1},b_j[
\\ 
  \partial_x v^{(j)} + ik v^{(j)} ={}&
  \partial_x v^{(j-1)} + ik v^{(j-1)} ,
&& \text{at } x=b_{j-1}
\\ 
  - \partial_x v^{(j)} + ik v^{(j)} = {}& 0
&& \text{at } x=b_{j} .
\end{align}
Then by induction it follows that
\begin{equation} \label{eq:solution_v}
  v^{(j)}(x) = \frac{i}{2k} \int_0^x e^{ik(x-s)} f(s) \, ds
    + \frac{i}{2k} \int_x^{b_{j}} e^{-ik(x-s)} f(s) \, ds 
\end{equation}
for $b_{j-1} < x < b_j$. Indeed, if (\ref{eq:solution_v}) is satisfied
with $j$ replaced by $j-1$, it follows that
\begin{equation}
  \partial_x v^{(j-1)}(b_{j-1}) + ik v^{(j-1)}(b_{j-1})
  = \int_0^{b_{j-1}} e^{ik(b_{j-1} - s)} f(s) \, ds ,
\end{equation}
which together with (\ref{eq:solutionformula1D}) implies 
(\ref{eq:solution_v}).

Next let $U^{(j)}$ satisfy
\begin{align} 
  -\partial_{xx}^2  U^{(j)} -k^2 U^{(j)} = {}& f^{(j)} 
&& \text{for } x\in ]b_{j-1},b_j[
\\ 
  \partial_x U^{(j)} + ik U^{(j)} ={}&
  \partial_x v^{(j-1)} + ik v^{(j-1)} ,
&& \text{at } x=b_{j-1}
\\
  - \partial_x U^{(j)} + ik U^{(j)} = {}& 
  - \partial_x U^{(j+1)} + ik U^{(j+1)}
&& \text{at } x=b_{j} ,
\end{align}
obtained by setting $U^{(J)} = v^{(J)}$ and solving $U^{(j)}$ for 
$j=J-1,J-2\ldots, 1$ in that order. (This way of double
sweeping is slighly different from the one above.) Then 
by induction
\[
  - \partial_x U^{(j+1)}(b_j) + ik U^{(j+1)}(b_j)
  = \int_{b_j}^L e^{-ik(b_j-s)} f(s) \, ds 
\]
and for $b_{j-1} < x < b_j$
\[
  U^{(j)}(x) = \frac{i}{2k} \int_0^x e^{ik(x-s)} f(s) \, ds 
+ \frac{i}{2k} \int_x^L e^{-ik(x-s)} f(s) \, ds 
\]
i.e.\ the solution of the full problem.

\medskip

Next we discuss the case of an upward and a downward sweep that proceed 
concurrently. More precisely described it is the same method as in the 
introduction, but at iteration count $n$ only the functions $v^{(n)}_n$ and 
$v^{(J+1-n)}_n$ are updated.  Using the solution formula 
(\ref{eq:solutionformula1D}) again it can be verified
that $v_J(x)$, given by $v^{(j)}_J$ for $b_{j-1} < x < b_j$, is the
solution of the original problem.

\subsection{PML based transmission on the strip%
\label{ssec:theory_2D}}

Here we consider the problem with $k=\text{constant}$ on the strip
$]0,L[ \times ]0,1[$, with Dirichlet boundary conditions 
at $y=0$ and $y=1$ and PML boundary layers at $x=0$ and $x=L$.
In this section we assume that a PML boundary layer behaves like
a perfect non-reflecting boundary condition. 

The behavior of a perfect non-reflecting boundary is most easily
described in the Fourier domain. After a Fourier transform 
$u = \sum_l \sin(2 \pi l) \hat{u}_l(x)$,
$l=1,2,\ldots$, and writing $\hat{u}_l(x) =\hat{u}(x,\eta)$, $\eta =
2\pi l$, the Helmholtz equation becomes a family of ODE's that reads
\begin{equation} \label{eq:Fourier_transformed_PDE_2D}
  -\partial_{xx}^2 \hat{u} + \eta^2 \hat{u} - k^2 \hat{u} 
    = \hat{f}(x,\eta)
\end{equation}
We assume that $k \neq 2\pi l$ for all integers $l>0$.
The non-reflecting boundary condition becomes
\begin{align} \label{eq:perfectly_non_reflecting_bc1}
   \partial_x \hat{u} + \lambda \hat{u} = {}& h_1
&&
  \text{at $x=0$}
\\ \label{eq:perfectly_non_reflecting_bc2}
 - \partial_x \hat{u} + \lambda \hat{u} = {}& h_2
&&
  \text{at $x=L$,}
\end{align}
where $\lambda$ is given by
\begin{equation} \label{eq:perfectly_non_reflecting2}
 \lambda 
= \left\{ \begin{array}{ll}
  i \sqrt{k^2 - \eta^2} & \text{if } |\eta| < k \\
  - \sqrt{\eta^2 - k^2} & \text{if } |\eta| > k ,
\end{array}\right.
\end{equation}
and $h_1$ and $h_2$ are 0 for homogeneous non-reflecting boundary
conditions and non-zero if incoming waves are to be modeled.
(In the spatial domain, after inverse Fourier transform in $y$,
the factor $\lambda$ would become a pseudodifferential operator
that is non-local, explaining why in two and three dimension
we can not obtain the properties (i) and (ii) of the introduction
using Robin boundary conditions.)

We have the following result:

\begin{theorem}
In the situation just described, the map $P$ satisfies $A P f = f$.
\end{theorem}

\begin{proof}
The solution formula for 
(\ref{eq:Fourier_transformed_PDE_2D}-\ref{eq:perfectly_non_reflecting_bc2})
is given by
\[
\hat{u}(x,\eta) =
\frac{-1}{2\lambda} \int_0^x e^{\lambda (x-s)} \hat{f}(s,\eta) \, ds
  + \frac{-1}{2\lambda} \int_x^L e^{-\lambda(x-s)} \hat{f}(s,\eta) \, ds
  + \frac{e^{\lambda x}}{2\lambda} h_1
  + \frac{e^{-\lambda (x-L)}}{2\lambda} h_2
\]

First we consider the fields $v^{(j)}$, in other words the forward sweep.
Using induction, it easy to show that
\begin{equation} \label{eq:forward_sweep_strip1}
  \hat{v}^{(j)}(x,\eta) 
  =   \frac{-1}{2\lambda} \int_0^x e^{\lambda(x-s)} \hat{f}(s,\eta) \, ds
    + \frac{-1}{2\lambda} \int_x^{b_j} e^{-\lambda(x-s)} \hat{f}(s,\eta) \, ds .
\end{equation}
Indeed, assuming this is true with $j-1$ substituted for $j$ it follows 
that
\[
  \partial_x \hat{v}^{(j-1)} (b_{j-1},\eta) 
  = \frac{-1}{2} \int_0^x e^{\lambda (x-s)} \hat{f}(s,\eta) \, ds 
\]
The solution formula applied to right hand side
$\hat{f}^{(j)}(x,\eta) - 2 \delta(x-b_{j-1}) \partial_x 
\hat{v}^{(j-1)}(b_{j-1},\eta)$ 
then gives (\ref{eq:forward_sweep_strip1}).

Next we consider the backward sweep. The $w^{(j)}$ are solutions
to Helmholtz equations with as right hand side the residual $f - A v$ 
derived from the $v^{(j)}$. From (\ref{eq:forward_sweep_strip1}) it follows
that
\begin{equation} \label{eq:pf_robin_v1}
  (\partial_x + \lambda) \hat{v}(\tilde{b_j},\eta) 
  =  - \int_0^{\tilde{b}_j} e^{\lambda(x-s)} \hat{f}(s,\eta) \, ds .
\end{equation}
It follows that $\hat{v}+\hat{w}^{(j)}$ satisfies 
for $\tilde{b}_{j-1} < x < \tilde{b}_j$ the equations
\[
  - \partial^2_{xx} (\hat{v}+\hat{w}^{(j)}) + (\eta^2 - k^2)(\hat{v}+\hat{w}^{(j)}) 
   =  \hat{f}
\]
while at the boundaries of the interval
\begin{align} \label{eq:pf_induction_w2}
  (\partial_x + \lambda) (\hat{v}+\hat{w}^{(j)}) 
  = {}& - \int_0^{\tilde{b}_{j-1}} e^{\lambda(x-s)} \hat{f}(s,\eta) \, ds
&&
\text{at $x = \tilde{b}_{j-1}$}
\\ \label{eq:pf_induction_w3}
  (- \partial_x + \lambda) (\hat{v}+\hat{w}^{(j)}(\tilde{b}_{j}-0,\eta)) 
  = {}& (- \partial_x + \lambda) (\hat{v}+\hat{w}^{(j+1)}) 
&&
\text{at $x = \tilde{b}_{j}$}
\end{align}
Here $\hat{w}^{(j)}(\tilde{b}_{j}-0,\eta)$ denotes the limit
$\lim_{x \uparrow \tilde{b}_j} \hat{w}^{(j)}(x,\eta)$.
The first of these two equations follows easily from (\ref{eq:pf_robin_v1}),
while the second follows from the transmission condition.
Then by induction (\ref{eq:pf_induction_w3}) can also be written as
\[
  (- \partial_x + \lambda) (\hat{v}+\hat{w}^{(j)}(\tilde{b}_{j}-0,\eta)) 
  = - \int_{\tilde{b}_j}^L e^{-\lambda(x-s)} \hat{f}(s,\eta) \, ds 
\qquad\qquad
\text{at $x = \tilde{b}_{j}$}
\]
It follows that
\[
  \hat{v}(x,\eta) + \hat{w}^{(j)}(x,\eta)
  = 
\frac{-1}{2\lambda} \int_0^x e^{\lambda (x-s)} \hat{f}(s,\eta) \, ds
  + \frac{-1}{2\lambda} \int_x^L e^{-\lambda(x-s)} \hat{f}(s,\eta) \, ds
\]
for $\tilde{b}_{j-1} < x < \tilde{b}_j$, which completes the proof.
\end{proof}

\section{Numerical results}
\label{sec:numericalresults}

In this section we present examples in 2-D and in 3-D with constant
and variable $k$. We'll focus on the convergence of the method,
measured by the number of iterations for reduction of the residual by
a factor $10^{-6}$. 
After studying the method in its own right, we compare the method with
a method that combines classical Robin transmission conditions
with the double sweeping method presented here.

In our 2-D example we will vary the size of the
domain and the number of subdomains, keeping $h \omega$ constant. We 
will see that the number of iterations required is essentially
independent of those parameters.
In 3-D we take subdomains of constant thickness of 10 grid points,
excluding the PML layers. The number of subdomains is therefore
dictated by the size of the domain, and we study the convergence as a
function of domain size. Again the number of iterations is approximately
constant.  We
also study the influence on the parameter $w_{\rm pml}$ for constant
coefficient media. In our 3-D examples a value of $w_{\rm pml} = 4$ 
generally produced a good convergence. Nevertheless the parameter 
$w_{\rm pml}$ has some influence and some insight in this is obtained 
from the third example. The 3-D examples 
were done with domain sizes up to $(400)^3$.

Because of the size of the problems, the implementation was done under 
MPI. For the solution of the
linear systems on the subdomain the parallel sparse multifrontal
solver MUMPS \cite{AmestoyEtAl2001_MUMPS} was used.  In the version 
used to generate the 2-D examples the sequential sparse
multifrontal solver UMFPACK \cite{Davis2004_UMFPACK} was used.
The examples were run on the LISA linux cluster of the 
Stichting Academisch Rekencentrum Amsterdam (SARA).

The final part of this section concerns a comparison of PML-based and 
Robin transmission conditions. This is done in 2-D using a constant and 
a random medium. For these tests a Matlab implementation was used.

\subsection{Example 1: Marmousi}

Our first example is the Marmousi model, a synthetic model from
reflection seismology. In this model the velocity $c(x,y)$ varies 
between $1500$ and $5500$ ms${}^{-1}$. The model and a solution
to the Helmholtz equation are given in Figure~\ref{fig:marm_fig}.

Our first set of computations shows the number of iterations required 
for convergence as a function of grid size $h$ and the number of 
subdomains $J$. It is summarized in Table~\ref{tab:example1}.
The grid size varies between $h=1$ and $h=16$ m, and
the number of subdomains between 3 and 300. The frequency $\omega$ 
is chosen such that $h \omega$ is constant. The thickness of the PML 
layer is given by $w_{\rm pml} = 5$ except for the case with 300 subdomains 
which we simulated twice, with $w_{\rm pml}=5$ and $w_{\rm pml}=6$.

What stands out is that the convergence is very fast, with between
4 and 9 iterations required for reduction of the residual by 
$10^{-6}$. There is only a mild dependence on the grid size and
on the number of subdomains.
The dependence on $w_{\rm pml}$ and the somewhat larger number for 
300 subdomains with $w_{\rm pml}=5$ will be discussed below.

\begin{figure}[ht]
\begin{center}
(a)\\
\includegraphics[width=110mm]{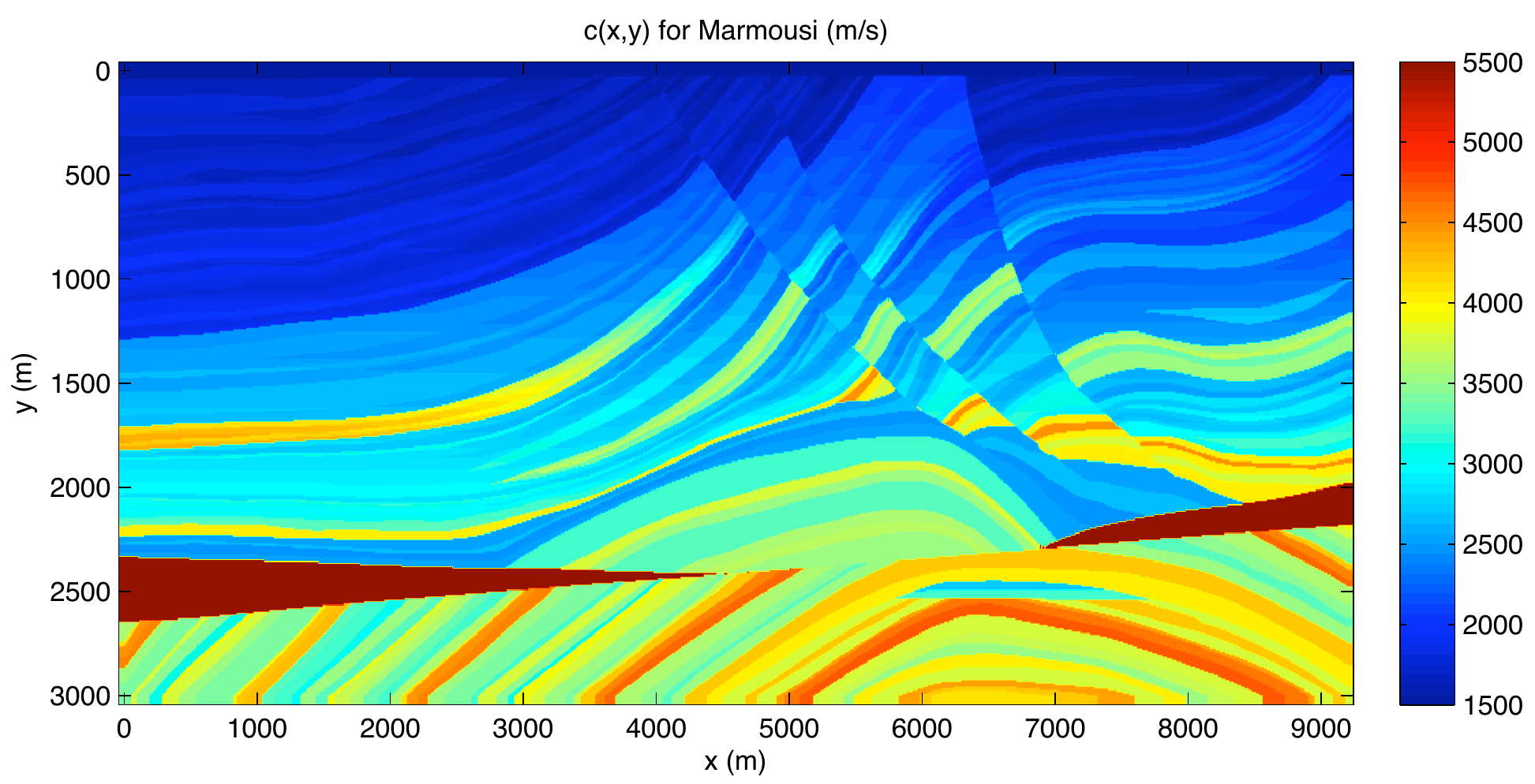}
\end{center}
\begin{center}
(b)\\
\includegraphics[width=110mm]{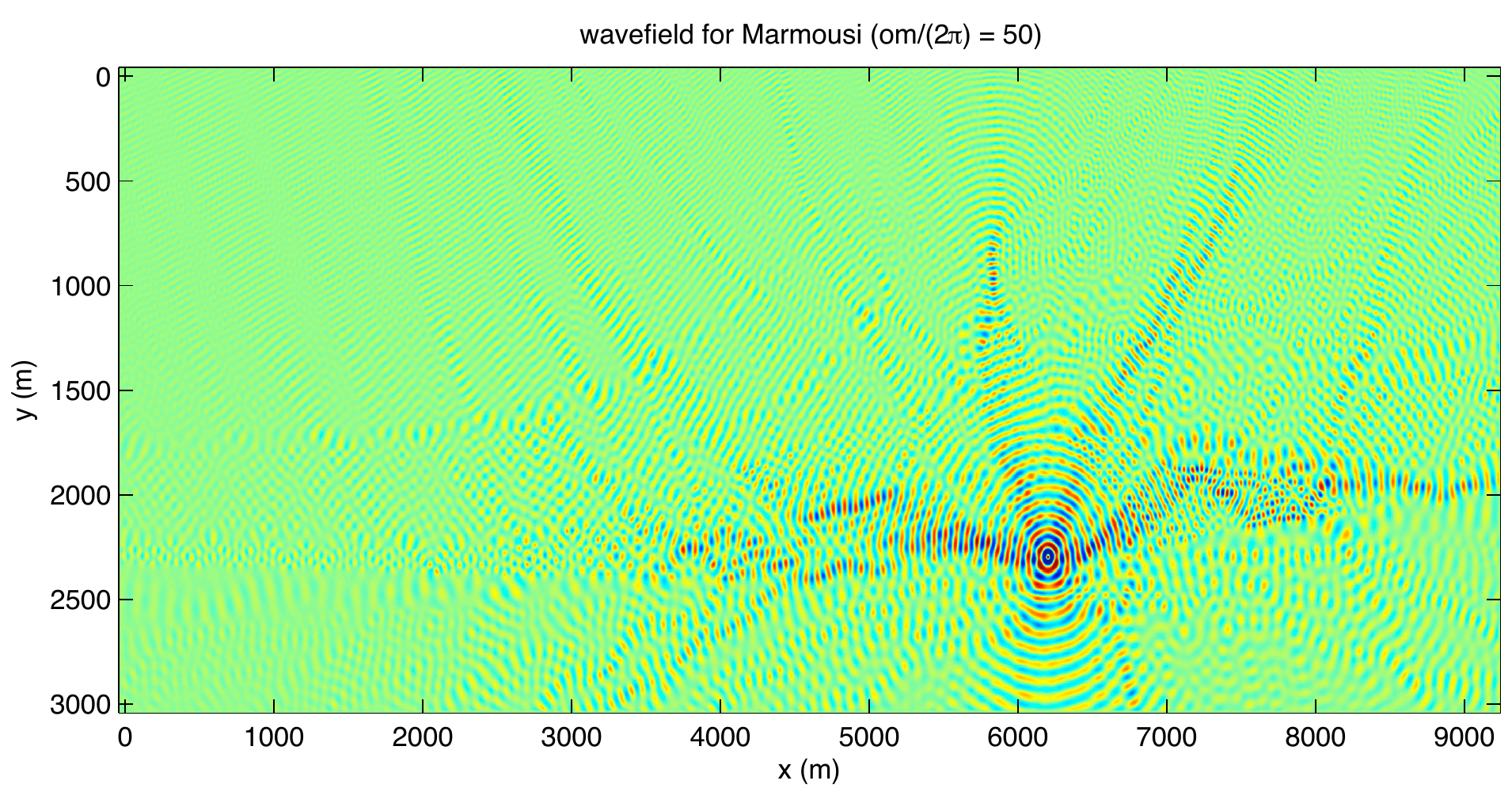}
\end{center}
\caption{Marmousi model and solution with $\frac{\omega}{2\pi} = 50$}
\label{fig:marm_fig}
\end{figure}

\begin{table}
\begin{center}
\begin{tabular}{|l|l|l||c|c|c|c|c|} \hline
\multirow{2}{*}{$N_x \times N_y$}
  & \multirow{2}{*}{$h$ (m)} 
  & \multirow{2}{*}{$\frac{\omega}{2 \pi}$ (Hz)}
  & \multicolumn{5}{c|}{Number of $x$-subdomains}
\\ \cline{4-8}
  & 
  & 
  & 3
  & 10 
  & 30
  & 100
  & 300
\\ \hline&&&&&&&\\[-2.41ex]\hline
$600 \times 212$ 
  & 16 
  & 12.5
  & 4
  & 5
  & 6
  &
  &
\\ \hline
$1175 \times 400$ 
  & 8
  & 25 
  & 5
  & 6
  & 7
  &
  &
\\ \hline
$2325 \times 775$
  & 4 
  & 50
  & 6
  & 6
  & 7 
  & 9
  & 
\\ \hline
$4625 \times 1525$
  & 2
  & 100
  & 6
  & 6 
  & 7 
  & 8 
  &
\\ \hline
$9225 \times 3025$
  & 1
  & 200
  & 
  & 7
  & 8
  & 9 
  & 13 (8) (*)
\\ \hline  
\end{tabular}\\
(*) $13$ was obtained for $w_{\rm pml}=5$, $8$ for $w_{\rm pml} = 6$.\\
\end{center}
\caption{Convergence results for example 1.
Displayed is the number of iterations for reduction of the 
residual by $10^{-6}$ as a function of the size of the domain
and the number of subdomains.}
\label{tab:example1}
\end{table}

\subsection{Example 2: A random medium in 3-D}

Our second example is random medium in 3-D. Plots of the medium and
a solution are given in Figure~\ref{fig:rand3_fig}. The size of the example
varied between $100^3$ and $400^3$ (excluding PML layers on the sides). 
In all cases the medium was divided in layers of thickness $10$ (excluding
again the PML layers). Experiments were performed with
$w_{\rm pml} = 4$ and $5$. The thickness of the subdomain on which the 
computation took place was hence 19 and 21 grid points respectively. 
The results are summarized in Table~\ref{tab:example2}.

The result are similar to those of the Marmousi examples. The iterative
method converged rapidly, in 6 to 8 iterations. In these examples the 
value $w_{\rm pml} = 4$ is sufficient.

\begin{figure}[ht]
\begin{center}
(a)\\
\includegraphics[width=88mm]{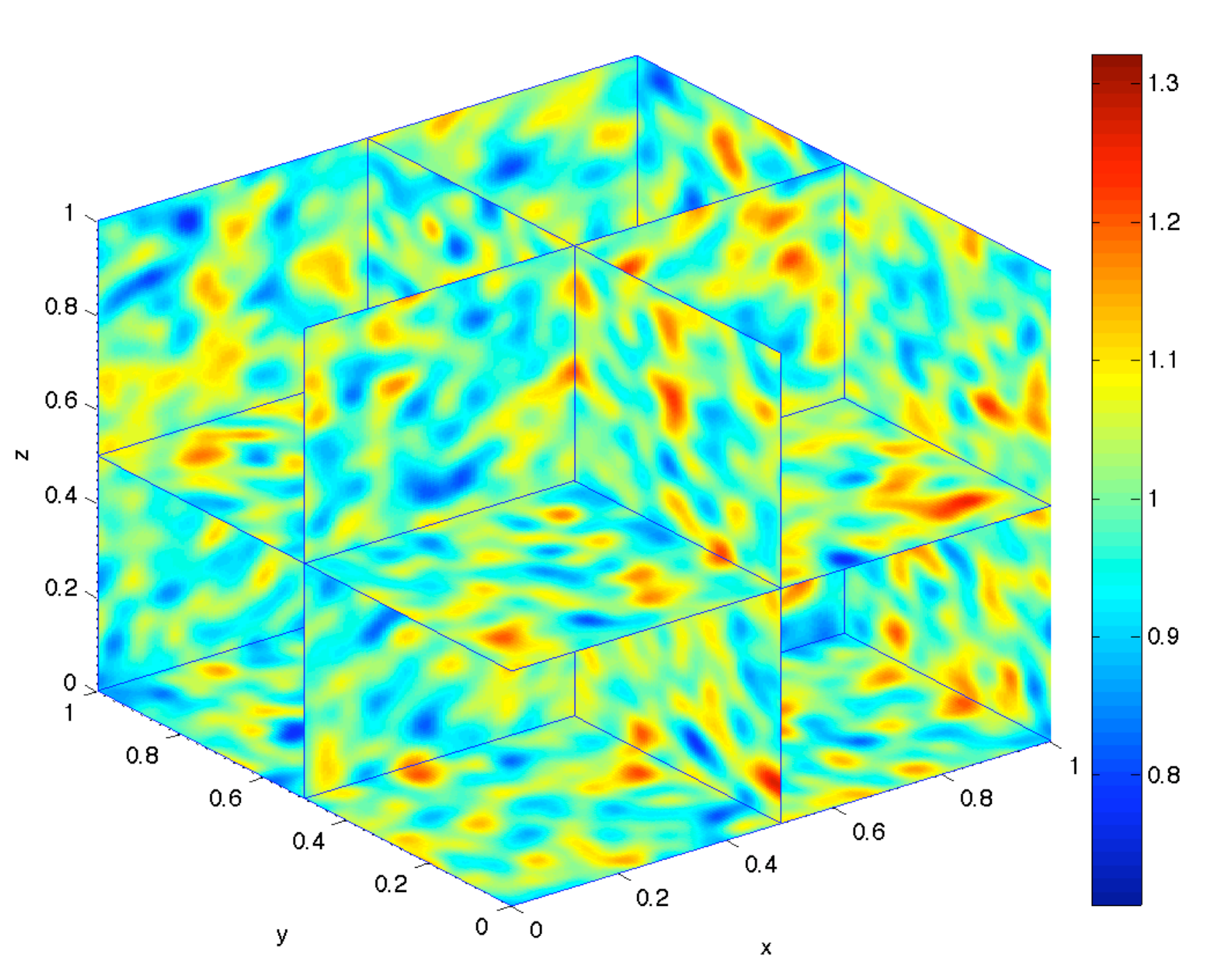}
\end{center}
\begin{center}
(b)\\
\includegraphics[width=80mm]{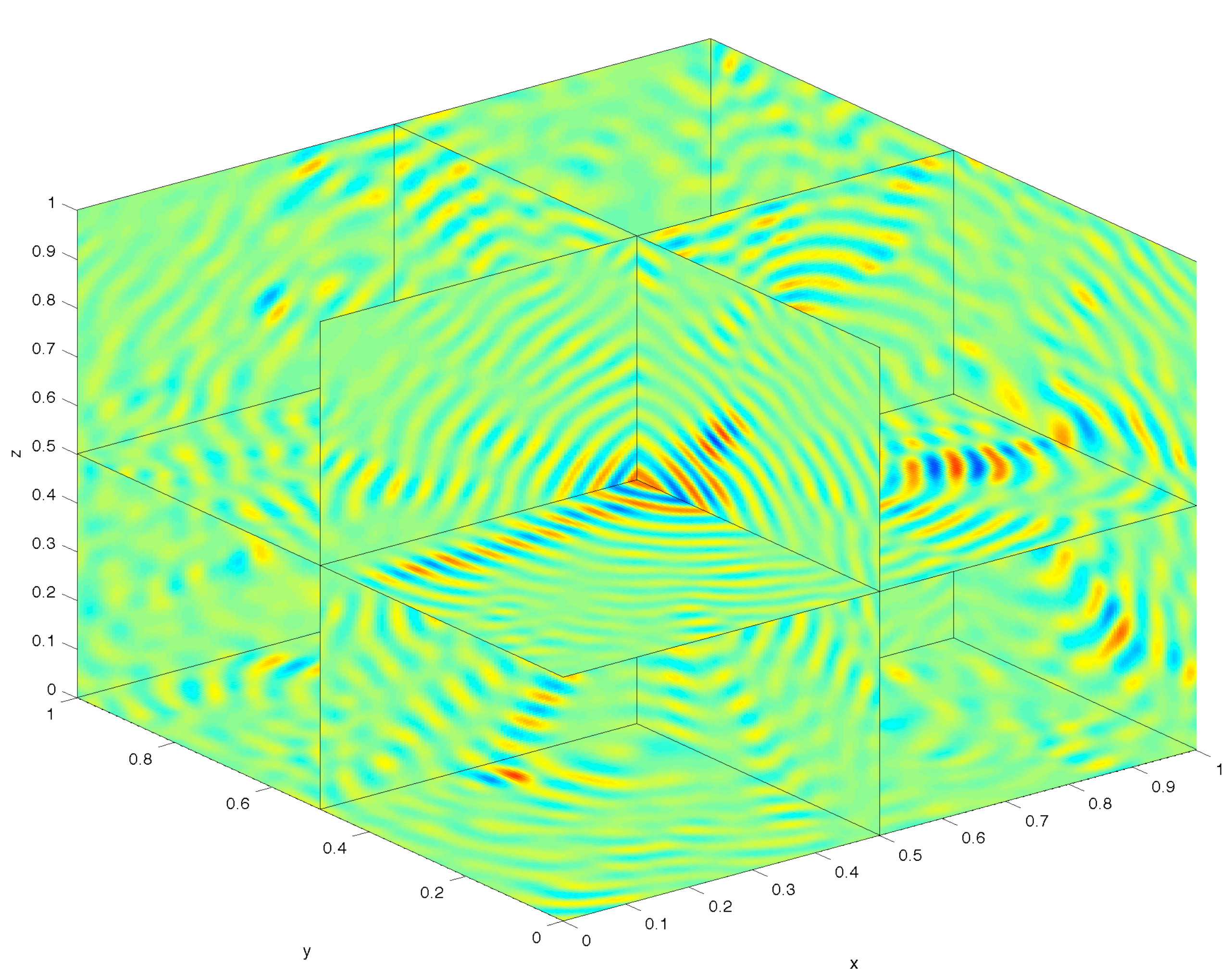}
\end{center}
\caption{(a) Random medium used in example 2; (b) solution to the Helmholtz
equation with a point source.}
\label{fig:rand3_fig}
\end{figure}

\begin{table}
\begin{center}
\begin{tabular}{|l|l|l|l||c|c|} \hline
\multirow{2}{*}{$n_x \times n_y \times n_z$}
  & \multirow{2}{*}{$h$} 
  & \multirow{2}{*}{$\frac{\omega}{2 \pi}$}
  & \multirow{2}{*}{$J$}
  & \multicolumn{2}{c|}{$w_{\rm pml}$}
\\ \cline{5-6}
  & 
  & 
  &
  & 4
  & 5
\\ \hline&&&&\\[-2.41ex]\hline
$100 \times 100 \times 100$
  & 0.01 
  & 10
  & 10
  & 6
  & 5
\\ \hline
$200 \times 200 \times 200$
  & 0.005
  & 20
  & 20
  & 6
  & 6
\\ \hline
$300 \times 300 \times 300$
  & 0.00333
  & 30
  & 30
  & 7
  & 6
\\ \hline
$400 \times 400 \times 400$
  & 0.0025
  & 40
  & 40
  & 8
  & 6
\\ \hline
\end{tabular}
\end{center}
\caption{Convergence results for example 2, a random medium in 3-D.}
\label{tab:example2}
\end{table}

\subsection{Example 3: A constant medium in 3-D and varying $w_{\rm pml}$}

In our third example we explore the dependence of the 
convergence on $w_{\rm pml}$. The main conclusion of the previous two
examples is that convergence is fast in all cases. Nevertheless, 
a increase in $w_{\rm pml}$ reduces the number of iterations somewhat
in the larger examples. 

In this example the domain is the unit cube, and the velocity $c =
1$. (A constant medium is attractive because it requires less
computational resources, due to the fact that for only one
subdomain the $L D L^t$ decomposition has to be computed.) The
subdomain size varies between $100^3$ and $400^3$ (excluding the outer
PML layers), while the thickness of the PML layers varies between 3
and 6 gridpoints. The frequency $\omega$ is chosen to correspond to 
10 grid points per wavelength. The results are given in 
Table~\ref{tab:example3}.


While we have limited data, still the following pattern can be observed.
For fixed $w_{\rm pml}$ the number of iterations increases with
the grid size. However the number of iterations can be kept more
or less constant if one can increases $w_{\rm pml}$ at the same time
as the grid size. Here $w_{\rm pml}$ goes roughly logarithmically with 
the grid size. 
\begin{table}
\begin{center}
\begin{tabular}{|l|l|l|l||c|c|c|c|} \hline
\multirow{2}{*}{$n_x \times n_y \times n_z$}
  & \multirow{2}{*}{$h$} 
  & \multirow{2}{*}{$\frac{\omega}{2 \pi}$}
  & \multirow{2}{*}{$J$}
  & \multicolumn{4}{c|}{$w_{\rm pml}$}
\\ \cline{5-8}
  & 
  & 
  &
  & 3
  & 4
  & 5
  & 6
\\ \hline&&&&&&\\[-2.41ex]\hline
$100 \times 100 \times 100$
  & 0.01 
  & 10
  & 10
  & 5
  & 4
  & 4
  & 3
\\ \hline
$200 \times 200 \times 200$
  & 0.005
  & 20
  & 20
  & 7
  & 5
  & 4
  & 4
\\ \hline
$400 \times 400 \times 400$
  & 0.0025
  & 40
  & 40
  & 10
  & 7
  & 5
  & 5
\\ \hline
\end{tabular}
\end{center}

\caption{Convergence results for example 3.
Displayed is the number of iterations for reduction of the 
residual by $10^{-6}$ as a function of domain size and $w_{\rm pml}$.}
\label{tab:example3}
\end{table}

\subsection{Comparison between Robin and PML-based transmission
  conditions}

Motivated by our results so far, we study a double sweep method 
with Robin transmission conditions. This appears to be a new 
combination even though Robin transmission conditions have 
been extensively studied. We will compare this with the method above.

Similarly as above, we introduce overlapping subintervals of the 
$x$-axis, here denoted by
$]l^{(j)},r^{(j)}[$, $j=1,\ldots,J$, with 
$r^{(j)} = l^{(j+1)} + m_{\rm overlap} h$, 
$m_{\rm overlap}$ denoting the overlap in gridpoints
(i.e.\ $l^{(j)} = b_j$ and $r^{(j)} = \tilde{b}^{(j+1)}$).
In 2-D, for a rectangular domain $]0,L[ \times ]0,L_y[$, the right
sweep with Robin transmission conditions amounts to solving the
boundary value problems
\begin{align*}
  -\partial_{xx}^2 v^{(j)} - \partial_{yy}^2 v^{(j)}
  - k(x)^2 v^{(j)} = {}& f^{(j)}
&& 
  \text{for $l^{(j)} < x < r^{(j)}$, $0<y<L_y$} 
\\
  \partial_x v^{(j)} + ik v^{(j)} = {}& \partial_x v^{(j-1)} + ik v^{(j-1)}
&&
  \text{at $x = l^{(j)}$, $0<y<L_y$}
\\
  - \partial_x v^{(j)} + ik v^{(j)} = {}& 0
&&
  \text{at $x = r^{(j)}$, $0<y<L_y$}
\end{align*}
for $j=1,\ldots, J$ consecutively, where $f^{(j)}(x,y) = f(x,y)$ for 
$l^{(j)} < x < l^{(j+1)}$ and zero elsewhere
and PML modifications are assumed to be present near all the
external boundaries. This results in a an approximate solution 
given by $v(x) = v^{(j)}(x)$ for $l^{(j)} < x < l^{(j+1)}$.
The left sweep uses the residual $g = f - Au$ as right hand side
and is otherwise a left-right reflection of the right sweep.
This algorithm was implemented in Matlab.

A choice in this algorithm was the overlap parameter 
$m_{\rm overlap}$. This parameter was set equal to 1, since a zero 
overlap resulted in significantly worse convergence and larger
overlaps did not significantly improve the convergence.

On the unit square tests were performed for a constant medium 
($c =1$) and a random medium displayed in Figure~\ref{fig:rand2d}. 
We chose $N_x$ ranging from 100 to 1600 and $N_y = N_x$. 
The layer thickness was set at 10 points. Because of the absence
of PML layers, the subdomain solves are roughly 4 times cheaper when
using Robin transmission conditions compared to PML based conditions.
 Iteration numbers for reduction of the residual by $10^{-6}$
are given in Tables~\ref{tab:example_comparison1}
and~\ref{tab:example_comparison2}.
\begin{figure}[ht]
\begin{center}
\includegraphics[width=80mm]{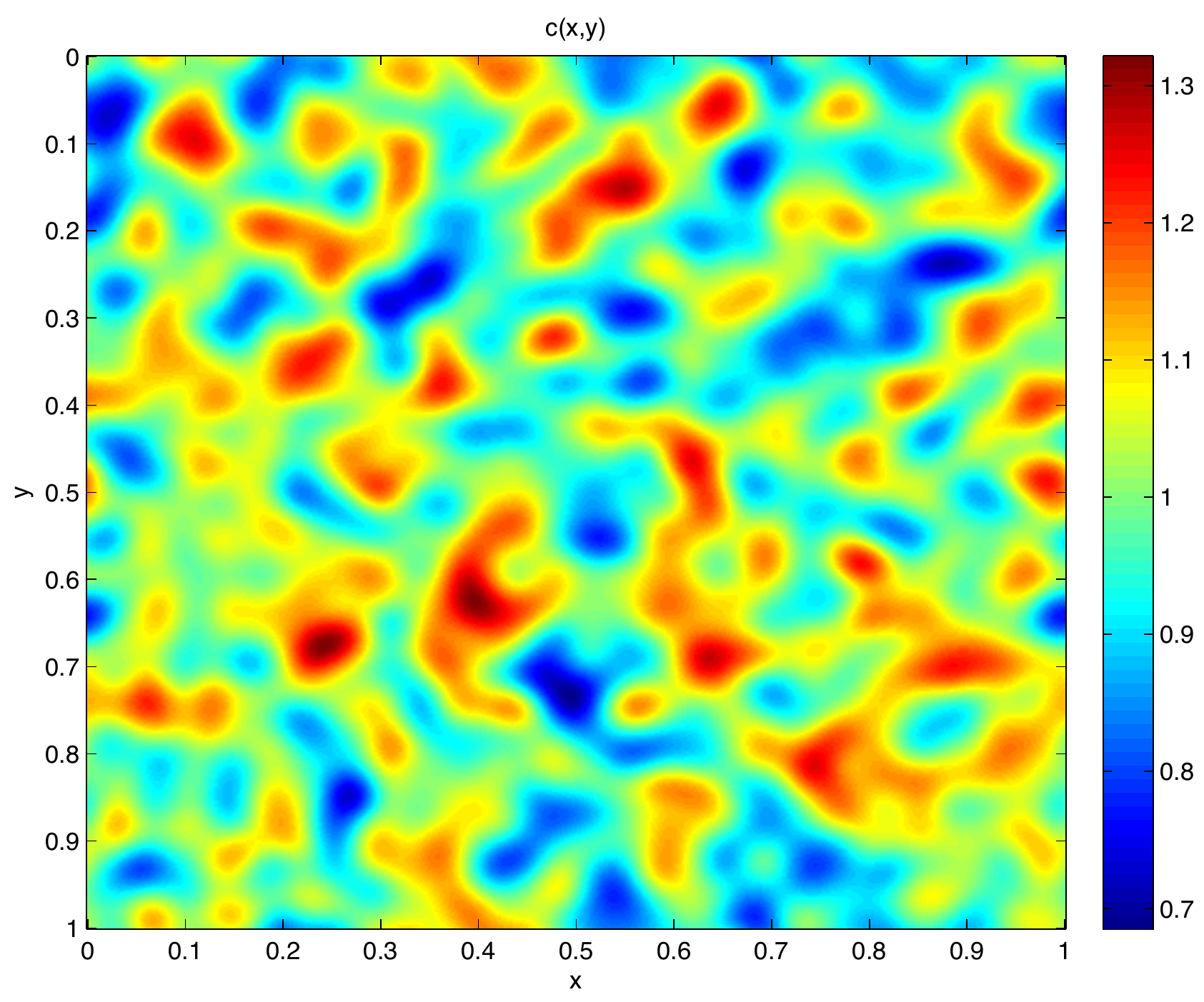}
\end{center}
\caption{Random medium used for the comparision of Robin and 
PML-based transmission conditions}
\label{fig:rand2d}
\end{figure}
\begin{table}[ht]
\begin{center}
\begin{tabular}{|l|l|l|l || c | c |} \hline
$N_x \times N_y$
  & $h$ 
  & $\frac{\omega}{2 \pi}$
  & $J$
  & PML
  & Robin
\\ \hline&&&&\\[-2.41ex]\hline
$100 \times 100$
  & 0.01 
  & 10
  & 10
  & 3
  & 9
\\ \hline
$200 \times 200$
  & 0.005
  & 20
  & 20
  & 4
  & 13
\\ \hline
$400 \times 400$
  & 0.0025
  & 40
  & 40
  & 4
  & 20
\\ \hline
$800 \times 800$
  & 0.00125
  & 80
  & 80
  & 5
  & 42
\\ \hline
$1600 \times 1600$
  & 0.000625
  & 160
  & 160
  & 7
  & 103
\\ \hline
\end{tabular}\\[1ex]
\end{center}
\caption{Comparison of convergence between Robin and PML-based 
transmission conditions for a constant medium.}
\label{tab:example_comparison1}
\end{table}
\begin{table}
\begin{center}
\begin{tabular}{|l|l|l|l || c | c |} \hline
$N_x \times N_y$
  & $h$ 
  & $\frac{\omega}{2 \pi}$
  & $J$
  & PML
  & Robin
\\ \hline&&&&\\[-2.41ex]\hline
$100 \times 100$
  & 0.01 
  & 7.14
  & 10
  & 7
  & 11
\\ \hline
$200 \times 200$
  & 0.005
  & 14.29
  & 20
  & 6
  & 14
\\ \hline
$400 \times 400$
  & 0.0025
  & 28.57
  & 40
  & 6
  & 20
\\ \hline
$800 \times 800$
  & 0.00125
  & 57.14
  & 80
  & 7
  & 34
\\ \hline
$1600 \times 1600$
  & 0.000625
  & 114.3
  & 160
  & 8
  & 74
\\ \hline
\end{tabular}
\end{center}
\caption{Comparison of convergence between Robin and PML-based 
transmission conditions for the random medium displayed in
Figure~\ref{fig:rand2d}.}
\label{tab:example_comparison2}
\end{table}

Two conclusions can be drawn.  First the method looks very
interesting, and certainly seems worthy of further study. On the other
hand the remarkable scaling of the PML-based transmission conditions
is not reproduced. With the Robin transmission conditions the
iteration numbers grow roughly linearly in $N_x$, or as $N^{1/2}$ in
2-D.  In 3-D this would lead to iteration numbers $O(N^{1/3})$.

We thank one of the 
anonymous reviewers for suggesting a comparison with Robin 
transmission conditions.

\section{Discussion}
\label{sec:conclusions}

A new domain decomposition method for the Helmholtz equation was presented.
It has remarkably fast convergence, even in the case of thin-layered 
subdomains. We have focussed on the use of this method with sparse
direct solvers on the subdomains. 

The method is related to that of Sch\"adle et al.\ in
\cite{SchadleEtAl2007}. In this reference, the authors consider finite
element methods for the time harmonic Maxwell equations on unbounded
domains truncated using the perfectly matched layer. A domain
decomposition method using PML-based interface conditions is derived
using a single sweep in each iteration. While the transmission term is
different from the one derived here, the difference is not very
relevant since its contribution propagates from the boundary $b_j$
directly into the PML layer, not entering the physical
domain. Numerical results are given for a 2-D example, using 2 or 3
subdomains, where in the second case the convergence is markedly
worse, probably due to the use of a single sweep. We conclude that
using a {\em double} sweep preconditioner is essential to obtain the
good convergence properties.

As pointed out in the introduction, the use of multiplicative domain
decomposition implies that the method is by nature sequential.  There
are basically two ways to obtain good parallel performance. One is the
parallellization of the $LDL^t$ factorization and backsubstitution
steps. Such an approach is described in
\cite{PoulsonEtAl2012_Preprint} for the sweeping preconditioner. This
is mostly a problem of parallel linear algebra, and not of domain
decomposition (although the distribution of the unknowns is relevant
for both parts of the story).  The second strategy is to divide the
subdomains over groups of processing nodes and perform the computation
for multiple right hand sides in a pipelined fashion. Because of the
setup time, the method is most relevant for the case with multiple
right hand sides anyway. (In other cases it probably makes more sense
to opt e.g.\ for the shifted Laplacian method).


The solutions to the time harmonic Maxwell equations and the time harmonic
linear elastic wave equation behave in many respects the same as those
of the Helmholtz equation. We expect that the techniques outlined
in this paper are applicable in those cases as well.

\section*{References}

\bibliographystyle{model1-num-names}
\bibliography{helmdd}

\end{document}